\documentclass[11pt,twoside,a4paper]{article}
\usepackage{amscd}
\usepackage{amsmath}
\usepackage{amssymb}
\usepackage{latexsym}
\usepackage{stmaryrd}
\usepackage[latin1]{inputenc}
\usepackage[T1]{fontenc}   
\usepackage[english]{babel}
\newcommand{\tun}{\begin{picture}(5,0)(-2,-1)
\put(0,0){\circle*{2}}
\end{picture}}

\newcommand{\tdeux}{\begin{picture}(7,7)(0,-1)
\put(3,0){\circle*{2}}
\put(3,0){\line(0,1){5}}
\put(3,5){\circle*{2}}
\end{picture}}

\newcommand{\ttroisun}{\begin{picture}(15,12)(-5,-1)
\put(3,0){\circle*{2}}
\put(-0.65,0){$\vee$}
\put(6,7){\circle*{2}}
\put(0,7){\circle*{2}}
\end{picture}}
\newcommand{\ttroisdeux}{\begin{picture}(5,15)(-2,-1)
\put(0,0){\circle*{2}}
\put(0,0){\line(0,1){5}}
\put(0,5){\circle*{2}}
\put(0,5){\line(0,1){5}}
\put(0,10){\circle*{2}}
\end{picture}}

\newcommand{\tquatreun}{\begin{picture}(15,12)(-5,-1)
\put(3,0){\circle*{2}}
\put(-0.65,0){$\vee$}
\put(6,7){\circle*{2}}
\put(0,7){\circle*{2}}
\put(3,7){\circle*{2}}
\put(3,0){\line(0,1){7}}
\end{picture}}
\newcommand{\tquatredeux}{\begin{picture}(15,18)(-5,-1)
\put(3,0){\circle*{2}}
\put(-0.65,0){$\vee$}
\put(6,7){\circle*{2}}
\put(0,7){\circle*{2}}
\put(0,14){\circle*{2}}
\put(0,7){\line(0,1){7}}
\end{picture}}
\newcommand{\tquatretrois}{\begin{picture}(15,18)(-5,-1)
\put(3,0){\circle*{2}}
\put(-0.65,0){$\vee$}
\put(6,7){\circle*{2}}
\put(0,7){\circle*{2}}
\put(6,14){\circle*{2}}
\put(6,7){\line(0,1){7}}
\end{picture}}
\newcommand{\tquatrequatre}{\begin{picture}(15,18)(-5,-1)
\put(3,5){\circle*{2}}
\put(-0.65,5){$\vee$}
\put(6,12){\circle*{2}}
\put(0,12){\circle*{2}}
\put(3,0){\circle*{2}}
\put(3,0){\line(0,1){5}}
\end{picture}}
\newcommand{\tquatrecinq}{\begin{picture}(9,19)(-2,-1)
\put(0,0){\circle*{2}}
\put(0,0){\line(0,1){5}}
\put(0,5){\circle*{2}}
\put(0,5){\line(0,1){5}}
\put(0,10){\circle*{2}}
\put(0,10){\line(0,1){5}}
\put(0,15){\circle*{2}}
\end{picture}}

\newcommand{\ptroisun}{\begin{picture}(15,12)(-5,-1)
\put(3,7){\circle*{2}}
\put(-0.65,0){$\wedge$}
\put(6,0){\circle*{2}}
\put(0,0){\circle*{2}}
\end{picture}}

\newcommand{\pquatreun}{\begin{picture}(15,12)(-5,-1)
\put(3,7){\circle*{2}}
\put(-0.65,0){$\wedge$}
\put(6,0){\circle*{2}}
\put(0,0){\circle*{2}}
\put(3,0){\circle*{2}}
\put(2.9,0){\line(0,1){7}}
\end{picture}}
\newcommand{\pquatredeux}{\begin{picture}(15,18)(-5,-1)
\put(3,14){\circle*{2}}
\put(-0.65,7){$\wedge$}
\put(6,7){\circle*{2}}
\put(0,7){\circle*{2}}
\put(0,0){\circle*{2}}
\put(0,0){\line(0,1){7}}
\end{picture}}
\newcommand{\pquatretrois}{\begin{picture}(15,18)(-5,-1)
\put(3,14){\circle*{2}}
\put(-0.65,7){$\wedge$}
\put(6,7){\circle*{2}}
\put(0,7){\circle*{2}}
\put(6,0){\circle*{2}}
\put(6,0){\line(0,1){7}}
\end{picture}}
\newcommand{\pquatrequatre}{\begin{picture}(15,18)(-5,-1)
\put(3,7){\circle*{2}}
\put(-0.65,0){$\wedge$}
\put(6,0){\circle*{2}}
\put(0,0){\circle*{2}}
\put(3,12){\circle*{2}}
\put(3,7){\line(0,1){5}}
\end{picture}}
\newcommand{\pquatrecinq}{\begin{picture}(15,12)(-5,-1)
\put(0,0){\circle*{2}}
\put(7,0){\circle*{2}}
\put(0,7){\circle*{2}}
\put(7,7){\circle*{2}}
\put(0,0){\line(0,1){7}}
\put(7,0){\line(0,1){7}}
\put(.5,1.5){$\scriptstyle \diagup$}
\end{picture}}
\newcommand{\pquatresix}{\begin{picture}(15,12)(-5,-1)
\put(0,0){\circle*{2}}
\put(7,0){\circle*{2}}
\put(0,7){\circle*{2}}
\put(7,7){\circle*{2}}
\put(0,0){\line(0,1){7}}
\put(7,0){\line(0,1){7}}
\put(0,1.5){$\scriptstyle \diagdown$}
\end{picture}}
\newcommand{\pquatresept}{\begin{picture}(15,12)(-5,-1)
\put(0,0){\circle*{2}}
\put(7,0){\circle*{2}}
\put(0,7){\circle*{2}}
\put(7,7){\circle*{2}}
\put(0,0){\line(0,1){7}}
\put(7,0){\line(0,1){7}}
\put(.5,1.5){$\scriptstyle \diagup$}
\put(0,1.5){$\scriptstyle \diagdown$}
\end{picture}}
\newcommand{\pquatrehuit}{\begin{picture}(15,18)(-5,-1)
\put(3,0){\circle*{2}}
\put(-0.65,0){$\vee$}
\put(6,7){\circle*{2}}
\put(0,7){\circle*{2}}
\put(3,14){\circle*{2}}
\put(-0.65,7){$\wedge$}
\end{picture}}

\input{xy}
\xyoption{all}

\newcommand{\PP}{\mathcal{PP}}
\newcommand{\PF}{\mathcal{PF}}
\newcommand{\h}{{\cal H}}
\renewcommand{\S}{\mathfrak{S}}
\newcommand{\prodg}{\rightsquigarrow}
\newcommand{\prodh}{\lightning}
\newcommand{\tdelta}{\tilde{\Delta}}

\title{Bruhat order on plane posets and applications}
\date{}
\author{Lo{\"\i}c Foissy \\
\\
{\small{\it Laboratoire de Mathématiques, Université de Reims}}\\
\small{{\it Moulin de la Housse - BP 1039 - 51687 REIMS Cedex 2, France}}\\
\small{e-mail : loic.foissy@univ-reims.fr}}

\newtheorem{defi}{\indent Definition}
\newtheorem{lemma}[defi]{\indent Lemma}
\newtheorem{cor}[defi]{\indent Corollary}
\newtheorem{theo}[defi]{\indent Theorem}
\newtheorem{prop}[defi]{\indent Proposition}

\newenvironment{proof}{{\bf Proof.}}{\hfill $\Box$}

\begin{document}

\maketitle

ABSTRACT. A plane poset is a finite set with two partial orders, satisfying a certain incompatibility condition. 
The set $\PP$ of isoclasses of plane posets owns two products, and an infinitesimal Hopf algebra structure is defined on the vector space
$\h_{\PP}$ generated by $\PP$, using the notion of biideals of plane posets.

We here define a partial order on $\PP$, making it isomorphic to the set of partitions  with the weak Bruhat order. 
We prove that this order is compatible with both products of $\PP$; 
moreover, it encodes a non degenerate Hopf pairing on the infinitesimal Hopf algebra $\h_{\PP}$. \\

Keywords. Plane posets; weak Bruhat order; infinitesimal Hopf algebras.\\

AMS classification. 06A11, 16W30, 06A07.

\tableofcontents

\section*{Introduction}

A \emph{double poset} is a finite set with two partial orders. The set of (isoclasses of) double posets owns several algebraic structures, as:
\begin{itemize}
\item  a product called \emph{composition}; it corresponds, roughly speaking, to the concatenation of Hasse graphs.
\item a coproduct, defined with the notion of \emph{ideals} for the first partial order. One obtains in this way the Malvenuto-Reutenauer 
Hopf algebra of double posets \cite{Reutenauer}.
\item a pairing defined with the helph of  Zelevinsky \emph{pictures}. It is shown in \cite{Reutenauer} that this pairing is Hopf;
consequently, the Hopf algebra of double posets is free, cofree and self-dual.
\end{itemize}
This Hopf algebra also contains many interesting subobjects, as, for example, the Hopf algebra of \emph{special posets},
that is to say double posets such that the second partial order is total, the Hopf algebra of plane posets \cite{Foissy2,Foissy4}, 
that is to say double posets such that the two partial orders satisfy an incompatibility condition (see definition \ref{1} below), 
or the noncommutative Hopf algebra of plane trees \cite{Foissy5,Foissy6,Holtkamp}, also known as the noncommutative Connes-Kreimer Hopf algebra. 
In particular, the Hopf subalgebra of plane poset turns out to be isomorphic to the Hopf algebra of permutations introduced by 
Malvenuto and Reutenauer in \cite{Reutenauer2}, also known ad the Hopf algebra of free quasi-symmetric functions \cite{Duchamp,Duchamp2}. 
An explicit isomorphism can be defined with the help of a bijection $\Psi_n$ between the set of plane posets on $n$ vertices and the symmetric group on $n$ letters, 
recalled here in theorem \ref{3}. This isomorphism and its applications are studied in \cite{Foissy2}. \\

We proceed here with the algebraic study of the links between permutations and plane posets. As the symmetric group $\S_n$ is partially ordered
by the weak Bruhat order, via the bijection $\Psi_n$ the set of plane posets is also partially ordered. This order has a nice combinatorial description,
see definition \ref{8}. It admits a decreasing bijection $\iota$, given by the exchange of the two partial orders defining plane posets; 
on the permutation side, this bijection consists of reversing the words representing the permutations. For example, let us give the Hasse graph
of this partial order restricted to plane posets of degree $3$, and the Hasse graph of the weak Bruhat order on $\S_3$:
$$
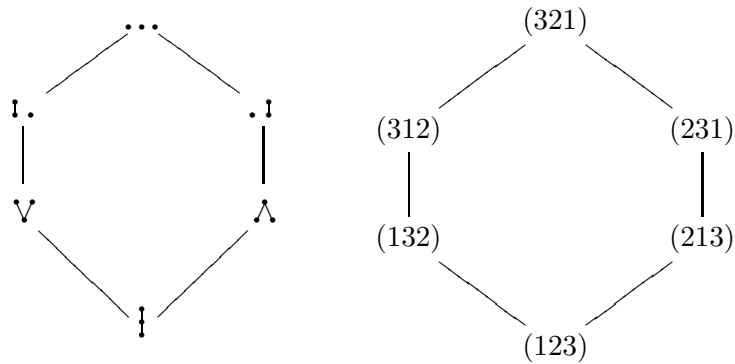
\xymatrix{&\tun\tun\tun\ar@{-}[dl]\ar@{-}[dr]&\\
\tdeux\tun\ar@{-}[d]&&\tun\tdeux\ar@{-}[d]\\
\ttroisun\ar@{-}[dr]&&\ptroisun\ar@{-}[dl]\\
&\ttroisdeux&}\hspace{1cm}\xymatrix{&(321)\ar@{-}[dl]\ar@{-}[dr]&\\
(312)\ar@{-}[d]&&(231)\ar@{-}[d]\\
(132)\ar@{-}[dr]&&(213)\ar@{-}[dl]\\
&(123)&}$$
This partial order is related to an infinitesimal Hopf algebra structure on plane posets. Recall that an infinitesimal Hopf algebra $\h$ \cite{Loday,Loday2}
is both an algebra and a coalgebra, satisfying the following compatibility: if $x,y\in \h$,
$$\Delta(xy)=\Delta(x)(1\otimes y)+(x\otimes 1)\Delta(y)-x\otimes y.$$
For a certain coproduct $\Delta_1$, given by biideals, the space of plane posets $\h_{\PP}$ becomes an infinitesimal Hopf algebra
for two products, the composition $m$ and the transformation $\prodh$ of $m$ by $\iota$. This coproduct is a special case of the four-parameters deformation 
of \cite{Foissy1}. This structure is also self-dual, with an explicit Hopf pairing $\langle-,-\rangle_1$ (theorem \ref{23}). This pairing is related 
to the partial Bruhat order in the following way: if $P,Q$ are two plane posets,
$$\langle P,Q\rangle_1=\left\{\begin{array}{l}
1 \mbox{ if }\iota(P)\leq Q,\\
0\mbox{ if not}.
\end{array}\right.$$
All these results admit a one parameter deformation, which is given in this text. \\

The text is organised as follows: the first section deals with double and plane posets: after some recalls, we give the definition of the infinitesimal coproduct and
its one-parameter deformation. The partial order on plane posets is defined in the second section; the isomorphism with the weak Bruhat order is also proved.
In the last section, the infinitesimal Hopf algebra structure and the partial order are related via the definition of a Hopf pairing.\\

{\bf Notations.} $K$ is a commutative field. All the vector spaces, algebras, coalgebras,$\ldots$ of this text are taken over $K$.

\section{Double and plane posets}

\subsection{Reminders}

\begin{defi}\label{1} \begin{enumerate}
\item \cite{Reutenauer} A \emph{ double poset} is a finite set $P$ with two partial orders $\leq_h$ and $\leq_r$.
\item A \emph{ plane poset} is a double poset $P$ such that for all $x, y\in P$, such that $x\neq y$,
$x$ and $y$ are comparable for $\leq_h$ if, and only if, $x$ and $y$ are not comparable for $\leq_r$.
The set of isoclasses of plane posets will be denoted by $\PP$. 
For all $n \in \mathbb{N}$, the set of isoclasses of plane posets of cardinality $n$ will be denoted by $\PP(n)$.
\item Let $P,Q \in \PP$. We shall say that $P$ is a \emph{plane subposet} of $Q$ if $P \subseteq Q$ and if the two partial orders of $P$ are the restriction
of the two partial orders of $Q$ to $P$.
\end{enumerate}\end{defi}

{\bf Examples}. Here are the plane posets of cardinal $\leq 4$. They are given by the Hasse graph of $\leq_h$; 
if $x$ and $y$ are two vertices of this graph which are not comparable for $\leq_h$, then $x \leq_r y$ if $y$ is more on the right than $x$.
\begin{eqnarray*}
\PP(0)&=&\{\emptyset\},\\
\PP(1)&=&\{\tun\},\\
\PP(2)&=&\{\tun\tun,\tdeux\},\\
\PP(3)&=&\{\tun\tun\tun,\tun\tdeux,\tdeux\tun,\ttroisun,\ttroisdeux,\ptroisun\},\\
\PP(4)&=&\left\{\begin{array}{c}
\tun\tun\tun\tun,\tun\tun\tdeux,\tun\tdeux\tun,\tdeux\tun\tun,\tun\ttroisun,\ttroisun\tun,\tun\ttroisdeux,\ttroisdeux\tun,\tun\ptroisun,\ptroisun\tun,\tdeux\tdeux,\\
\tquatreun,\tquatredeux,\tquatretrois,\tquatrequatre,\tquatrecinq,\pquatreun,\pquatredeux,\pquatretrois,\pquatrequatre,\pquatrecinq,\pquatresix,\pquatresept,\pquatrehuit
\end{array}\right\}.\end{eqnarray*}

The following proposition is proved in \cite{Foissy4} (proposition 11):

\begin{prop}
Let $P\in\PP$. We define a relation $\leq$ on $P$ by:
$$(x\leq y) \mbox{ if } (x\leq_h y\mbox{ or } x \leq_r y).$$
Then $\leq$ is a total order on $P$.
\end{prop}

As a consequence, for any plane poset $P\in \PP(n)$, we shall assume that $P=\{1,\ldots,n\}$ as a totally ordered set. \\

The following theorem is proved in \cite{Foissy2} (up to a passage to the inverse):

\begin{theo}\label{3}\begin{enumerate}
\item Let $\sigma\in \S_n$. We define a plane poset $P_\sigma$ in the following way:
\begin{itemize}
\item $P_\sigma=\{1,\ldots,n\}$ as a set.
\item If $i,j \in P_\sigma$, $i\leq_h j$ if $i\leq j$ and $\sigma^{-1}(i) \leq \sigma^{-1}(j)$.
\item If $i,j \in P_\sigma$, $i\leq_r j$ if $i\leq j$ and $\sigma^{-1}(i) \geq \sigma^{-1}(j)$.
\end{itemize}
The total order on $\{1,\ldots,n\}$ induced by this plane poset structure is the usual one.
\item For all $n \geq 0$, the following map is a bijection:
$$\Psi_n: \left\{\begin{array}{rcl}
\S_n&\longrightarrow&\PP(n)\\
\sigma&\longrightarrow&P_\sigma.
\end{array}\right.$$
\end{enumerate}\end{theo}

{\bf Examples.}
$$\begin{array}{rclccrclccrcl}
\Psi((1))&=&\tun,&&\Psi_2((12))&=&\tdeux,&&\Psi_2((21))&=&\tun\tun,\\
\Psi_3(123))&=&\ttroisdeux,&&\Psi_3((132))&=&\ttroisun,&&\Psi_3(213))&=&\ptroisun,\\
\Psi_3((231))&=&\tun\tdeux,&&\Psi_3(312))&=&\tdeux\tun,&&\Psi_3((321))&=&\tun\tun\tun,\\
\Psi_4(1234))&=&\tquatrecinq,&&\Psi_4((1243))&=&\tquatrequatre,&&\Psi_4((1342))&=&\tquatredeux,\\
\Psi_4((1324))&=&\pquatrehuit,&&\Psi_4((1423))&=&\tquatretrois,&&\Psi_4((1432))&=&\tquatreun,\\
\Psi_4((2134))&=&\pquatrequatre,&&\Psi_4((2143))&=&\pquatresept,&&\Psi_4((2314))&=&\pquatredeux,\\
\Psi_4((2341))&=&\ttroisdeux\tun,&&\Psi_4((2413))&=&\pquatrecinq,&&\Psi_4((2431))&=&\ttroisun\tun,\\
\Psi_4((3124))&=&\pquatretrois,&&\Psi_4((3142))&=&\pquatresix,&&\Psi_4((3214))&=&\pquatreun,\\
\Psi_4((3241))&=&\ptroisun\tun,&&\Psi_4((3412))&=&\tdeux\tdeux,&&\Psi_4((3421))&=&\tdeux\tun\tun,\\
\Psi_4((4123))&=&\tun\ttroisdeux,&&\Psi_4((4132))&=&\tun\ttroisun,&&\Psi_4((4213))&=&\tun\ptroisun,\\*
\Psi_4((4231))&=&\tun\tdeux\tun,&&\Psi_4((4312))&=&\tun\tun\tdeux,&&\Psi_4((4321))&=&\tun\tun\tun\tun.
\end{array}$$

We shall use three particular families of plane posets:

\begin{defi}
Let $P \in \PP$.
\begin{enumerate}
\item We shall say that $P$ is a \emph{plane forest} if it does not contain $\ptroisun$ as a plane subposet.
The set of plane forests is denoted by $\PF$.
\item We shall say that $P$ is a plane forest if it does not contain $\ptroisun$ as a plane subposet.
The set of plane forests will be denoted by $\PF$, and the set of plane forests of cardinality $n$ will be denoted by $\PF(n)$. \\
\end{enumerate}\end{defi}

{\bf Remark.} In other words, a plane poset is a plane forest if, and only if, its Hasse graph is a rooted forest.

\subsection{Algebraic structures on plane posets}

We define two products on $\PP$. The first is called \emph{composition} in \cite{Reutenauer} and is denoted by $\prodg$ in \cite{Foissy4}.
We shall shortly denote it by $m$ in this text.

\begin{defi}
Let $P,Q\in \PP$. 
\begin{enumerate}
\item The double poset $PQ$ is defined as follows:
\begin{itemize}
\item $PQ=P\sqcup Q$ as a set, and $P,Q$ are plane subposets of $PQ$.
\item For all $x \in P$, for all $y\in Q$, $x \leq_r y$.
\end{itemize}
\item The double poset $P\prodh Q$ is defined as follows:
\begin{itemize}
\item $P\prodh Q=P\sqcup Q$ as a set, and $P,Q$ are plane subposets of $PQ$.
\item For all $x \in P$, for all $y\in Q$, $x \leq_h y$.
\end{itemize}

\end{enumerate}\end{defi}

{\bf Examples.}\begin{enumerate}
\item The Hasse graph of $PQ$ is  the concatenation of the Hasse graphs of $P$ and $Q$. 
\item Here are examples for $\prodh$: $\tun \prodh \tdeux=\ttroisdeux,\: \tdeux \prodh \tun=\ttroisdeux,\: \tun \prodh \tun \tun=\ttroisun,\: \tun \tun \prodh \tun=\ptroisun$.
\end{enumerate}

The vector space generated by $\PP$ is denoted by $\h_{\PP}$. These two products are linearly extended to $\h_{\PP}$;
then $(\h_{\PP},.)$ and $(\h_{\PP},\prodh)$ are two associative, unitary algebras, sharing the same unit $1$, which is the empty plane poset.
Moreover, they are both graded by the cardinality of plane posets.

\subsection{Infinitesimal coproducts}

\begin{defi}
\cite{Reutenauer}. Let $P=(P,\leq_h,\leq_r)$ be a plane poset, and let $I \subseteq P$. 
\begin{enumerate}
\item We shall say that $I$ is a {\it $h$-ideal} of $P$, if, for all $x,y \in P$: $$(x\in I, \:x\leq_h y)\Longrightarrow (y\in I).$$
\item We shall say that $I$ is a {\it $r$-ideal} of $P$, if, for all $x,y \in P$: $$(x\in I, \:x\leq_r y)\Longrightarrow (y\in I).$$
\item We shall say that $I$ is a {\it biideal} of $P$ if it both a $h$-ideal and a $r$-ideal. 
\end{enumerate}\end{defi}

{\bf Remark.} If $P$ is a plane poset and $I\subseteq P$, $I$ is a biideal of $P$  if, for all $x,y \in P$:
$$(x\in I, \:x\leq y)\Longrightarrow (y\in I).$$

\begin{theo}
Let $q \in K$. We define a coproduct on $\h_{\PP}$ in the following way: for all $P\in \PP$,
$$\Delta_q(P)=\sum_{\mbox{\scriptsize $I$ biideal of $P$}} q^{h_{P\setminus I}^I}(P\setminus I) \otimes I,$$
where, for all $I,J \subseteq P$, $h_I^J=\{(x,y) \in I\times J\mid x <_h y\}$.
Then $\Delta_q$ is coassociative and for all $x,y\in \h_{\PP}$, using Sweedler notations:
\begin{eqnarray*}
\Delta_q(xy)&=&\sum x^{(1)}\otimes x^{(2)}y+\sum xy^{(1)}\otimes y^{(2)}-x \otimes y,\\
\Delta_q(x\prodh y)&=&\sum q^{|x^{(1)}||y|} x^{(1)} \otimes x^{(2)}\prodh y+\sum q^{|x||y^{(2)}|}x\prodh y^{(2)}\otimes y^{(1)}
-q^{|x||y|} x\otimes y.
\end{eqnarray*}
Hence, $(\h_{\PP},m,\Delta_q)$ is an infinitesimal Hopf algebra, as well as $(\h_{\PP},\prodh,\Delta_1)$.
\end{theo}

{\bf Notation.} For all nonempty $P\in \PP$, we put $\tdelta_q(P)=\Delta_q(P)-P\otimes 1-1\otimes P$. \\

\begin{proof} Let $P$ be a double poset. Let $I$ be a biideal of $P$ and let $J$ be a biideal of $I$; then $J$ is a biideal of $P$.
Let $I$ be a biideal of $P$ and let $J$ be a biideal of $P\setminus I$; then $I\sqcup J$ is a biideal of $P$. Hence:
\begin{eqnarray*}
(Id \otimes \Delta_q) \circ \Delta_q(P)&=&\sum_{\mbox{\scriptsize $J\subseteq I$ biideals of $P$}} 
q^{h_{P \setminus I}^I+h_{I\setminus J}^J}P\setminus (I \cup J) \otimes J\setminus I \otimes I,\\
(\Delta_q \otimes Id) \circ \Delta_q(P)&=&\sum_{\mbox{\scriptsize $J\subseteq I$ biideals of $P$}} 
q^{h_{P \setminus J}^J+h_{P\setminus I}^{I\setminus J}}P\setminus (I \cup J) \otimes J\setminus I \otimes I.
\end{eqnarray*}
Moreover:
$$h_{P \setminus I}^I+h_{I\setminus J}^J=h_{P \setminus I}^J+h_{P\setminus I}^{I\setminus J}+h_{I\setminus J}^J
=h_{P \setminus J}^J+h_{P\setminus I}^{I\setminus J},$$
so $\Delta_q$ is coassociative. Let us prove the compatibility of the products and the coproducts. We restrict ourselves to $x=P,y=Q \in \PP$. 
The result is obvious if $P=1$ or $Q=1$. Let us assume that $P,Q \neq 1$. The nontrivial biideals of $PQ$ are 
the nontrivial biideals of $Q$, $Q$, and the biideals $IQ$, where $I$ is a nontrivial biideal of $P$. Consequently:
\begin{eqnarray*}
\tdelta_q(PQ)&=&\sum_{\mbox{\scriptsize $I$ nontrivial biideal of $Q$}} q^{h_{PQ\setminus I}^I}PQ \setminus I\otimes I \\
&&+\sum_{\mbox{\scriptsize $I$ nontrivial biideal of $P$}} q^{h_{P\setminus I}^{IQ}}P \setminus I\otimes IQ+q^{h_P^Q}P\otimes Q\\
&=&\sum_{\mbox{\scriptsize $I$ nontrivial biideal of $Q$}} q^{0+h_{Q\setminus I}^I}PQ \setminus I\otimes I \\
&&+\sum_{\mbox{\scriptsize $I$ nontrivial biideal of $P$}} q^{h_{P\setminus I}^I+0}P \setminus I\otimes IQ+q^0 P\otimes Q\\
\Delta_q(PQ)&=&PQ \otimes 1+1\otimes PQ+(P \otimes 1) \tdelta_q(Q)+\tdelta_q(P)(1\otimes Q)+P\otimes Q\\
&=&PQ \otimes 1+1\otimes PQ+(P\otimes 1)\Delta_q(Q)-P\otimes Q-PQ\otimes 1\\
&&+\Delta_q(P)(1\otimes Q)-P\otimes Q-1\otimes PQ+P\otimes Q\\
&=&(P\otimes 1)\Delta_q(Q)-P\otimes Q+\Delta_q(P)(1\otimes Q)-P\otimes Q.
\end{eqnarray*}

The nontrivial ideals of $P\prodh Q$ are:
\begin{itemize} 
\item nontrivial ideals $I$ of $Q$. In this case,
$$h_{P\prodh Q\setminus I}^I=h_{Q\setminus I}^I+h_P^I=h_{Q\setminus I}^++|I|.|P|.$$
\item $Q$. In this case:
$$h_P^Q=|P|.|Q|.$$
\item ideals of the form $I \prodh Q$, where $I$ is a nontrivial ideal of $P$. In this case:
$$h_{P\setminus I}^{I\prodh Q}=h_{P\setminus I}^Q+h_{P\setminus I}^Q=h_{P\setminus I}^Q+|P\setminus I|.|Q|.$$
\end{itemize}
\begin{eqnarray*}
\tdelta_q(P\prodh Q)&=&\sum_{\mbox{\scriptsize $I$ nontrivial biideal of $Q$}} q^{|I|.|P|}q^{h_{PQ\setminus I}^I}P\prodh Q \setminus I\otimes I \\
&&+\sum_{\mbox{\scriptsize $I$ nontrivial biideal of $P$}} q^{|P\setminus I|.|Q|}q^{h_{P\setminus I}^{IQ}}P \setminus I\otimes I\prodh Q
+q^{|P|.|Q|}q^{h_P^Q}P\otimes Q\\
\Delta_q(PQ)&=&PQ \otimes 1+1\otimes PQ+\sum q^{|P|.|Q^{(1)}|}P\prodh Q^{(1)} \otimes Q^{(2)}-P\prodh Q \otimes 1-q^{|P|.|Q|}P\otimes Q\\
&&+\sum q^{|P^{(1)}|.|Q|}P^{(1)}\otimes P^{(2)}\prodh Q-1\otimes P\prodh Q-q^{|P|.|Q|}P\otimes Q+q^{|P|.|Q|}P\otimes Q\\
&=&\sum q^{|P|.|Q^{(1)}|}P\prodh Q^{(1)} \otimes Q^{(2)}
+\sum q^{|P^{(1)}|.|Q|}P^{(1)}\otimes P^{(2)}\prodh Q-q^{|P|.|Q|}P\otimes Q.
\end{eqnarray*}
In particular, if $q=1$, we recover the axiom of an infinitesimal Hopf algebra. \end{proof} \\

{\bf Remarks.} \begin{enumerate}
\item Obviously, both $(\h_{\PP},m,\Delta_q)$ and $(\h_{\PP},\prodh,\Delta_q)$ are graded by the cardinality of the double posets.
\item The coproduct $\Delta_q$ restricted to $\h_{\PP}$ is the coproduct $\Delta_{(q,0,1,0)}$ of \cite{Foissy1}.
\end{enumerate}

{\bf Examples.}
$$\begin{array}{rclcrcl}
\tdelta_q(\tdeux)&=&\tun\otimes \tun,&&
\tdelta_q(\tun\tun)&=&q\tun \otimes \tun,\\
\tdelta_q(\ttroisdeux)&=&\tun \otimes \tdeux+\tdeux \otimes \tun, &&
\tdelta_q(\ttroisun)&=&\tun \otimes \tun\tun+q\tdeux\otimes \tun,\\
\tdelta_q(\ptroisun)&=&q\tun \otimes \tdeux+\tun\tun \otimes \tun,&&
\tdelta_q(\tdeux\tun)&=&q\tun \otimes \tun\tun+q^2\tdeux \otimes \tun,\\
\tdelta_q(\tun\tdeux)&=&q^2\tun \otimes \tdeux+q\tun\tun \otimes \tun,&&
\tdelta_q(\tun\tun\tun)&=&q^2\tun \otimes \tun\tun+q^2\tun\tun\otimes \tun.
\end{array}$$

\section{Bruhat order on plane posets}

\subsection{Definition of the partial order}

\begin{defi}\label{8}
Let $P,Q$ be two plane posets of the same cardinality. We denote by $\theta_{P,Q}$ the unique increasing bijection (for the total order) from $P$ to $Q$.
\end{defi}

{\bf Remark.} If $P,Q,R$ are plane posets of the same cardinality, then obviously, $\theta_{Q,P}=\theta_{P,Q}^{-1}$
and $\theta_{P,R}=\theta_{Q,R}\circ \theta_{P,Q}$.

\begin{lemma}
Let $P,Q \in \PP(n)$. The following assertions are equivalent:
\begin{enumerate}
\item $\forall x,y \in P$, $(\theta_{P,Q}(x)\leq_h \theta_{P,Q}(y) $ in $Q)\Longrightarrow (x \leq_h y $ in $ P)$.
\item $\forall x,y\in P$, $(x \leq_r y$ in $P) \Longrightarrow (\theta_{P,Q}(x) \leq_r \theta_{P,Q}(y)$ in $Q)$.
\end{enumerate}
If this holds, we shall say that $P\leq Q$.
\end{lemma}

\begin{proof} $1\Longrightarrow 2$. Let us assume that $x \leq_r y$ in $P$. As $\theta_{P,Q}$ is increasing, $\theta_{P,Q}(x)\leq_h \theta_{P,Q}(y)$
or $\theta_{P,Q}(x)\leq_r \theta_{P,Q}(y)$ in $Q$. If $\theta_{P,Q}(x)\leq_h \theta_{P,Q}(y)$, by hypothesis, $x \leq h y$ in $P$.
As $P$ is a plane poset, $x=y$, so in both cases $\theta_{P,Q}(x)\leq_r \theta_{P,Q}(y)$. \\

$2\Longrightarrow 1$. Similar proof. \end{proof}

\begin{prop}
For all $n \geq 1$, the relation $\leq$ is a partial order on $\PP(n)$. 
\end{prop}

\begin{proof} Let us assume that $P\leq Q$ and $Q\leq P$. As $\theta_{Q,P}=\theta_{P,Q}^{-1}$, it satisfies the following assertion: 
$$\forall x,y \in P,\:(x \leq_h y \mbox{ in } P) \Longleftrightarrow (\theta_{P,Q}(x)\leq_h \theta_{P,Q}(y) \mbox{ in }Q).$$
Moreover, if $x \leq_r y$ in $P$, then, as $\theta_{P,Q}$ is increasing, $\theta_{P,Q}(x) \leq_h \theta_{P,Q}(y)$ 
or $\theta_{P,Q}(x)\leq_r \theta_{P,Q}(y)$ in $Q$. 
If $\theta_{P,Q}(x) \leq_h \theta_{P,Q}(y)$, then $x \leq_h y$ in $P$. By the incompatibility condition between $\leq_r$ and $\leq_h$, $x=y$,
so $\theta_{P,Q}(x)=\theta_{P,Q}(y)$, and $\theta_{P,Q}(x) \leq_r \theta_{P,Q}(y)$. 
In any case, $\theta_{P,Q}(x) \leq_r \theta_{P,Q}(y)$. Working with $\theta_{P,Q}^{-1}$, we obtain: 
$$\forall x,y \in P,\:(x \leq_r y \mbox{ in } P) \Longleftrightarrow (\theta_{P,Q}(x)\leq_r \theta_{P,Q}(y) \mbox{ in }Q).$$
So $\theta_{P,Q}$ is an isomorphism of plane posets. As a consequence, $P=Q$. \\

Let us assume that $P\leq Q$ and $Q\leq R$. As $\theta_{P,R}=\theta_{Q,R} \circ \theta_{P,Q}$, 
if $\theta_{P,R} (x) \leq_h \theta_{P,R}(y)$ in $R$, then $\theta_{P,Q}(x) \leq_h \theta_{P,Q}(y)$ 
in $Q$, so $x \leq_h y$ in $P$. So $P\leq R$.\\

Let $P\in \PP(n)$. The unique increasing bijection from $P$ to $P$ is $Id_P$, so, clearly, $P\leq P$. \end{proof}\\

{\bf Examples.} Here are the Hasse diagrams of $(\PP(2),\leq)$, $(\PP(3),\leq)$ and$(\PP(4),\leq)$:
$$\xymatrix{\tun \tun\ar@{-}[d]\\ \tdeux} \hspace{2cm}
\xymatrix{&\tun\tun\tun\ar@{-}[dl]\ar@{-}[dr]&\\
\tdeux\tun\ar@{-}[d]&&\tun\tdeux\ar@{-}[d]\\
\ttroisun\ar@{-}[dr]&&\ptroisun\ar@{-}[dl]\\
&\ttroisdeux&}$$
$$\xymatrix{&&&&&\tun\tun\tun\tun \ar@{-}[llld]\ar@{-}[rrrd]\ar@{-}[d]&&&&&\\
&&\tdeux\tun\tun\ar@{-}[ld] \ar@{-}[rrrd]|!{[rd];[rrr]}\hole &&&\tun\tdeux\tun\ar@{-}[lld]\ar@{-}[rrd]|!{[d];[rrr]}\hole&&&
\tun\tun\tdeux\ar@{-}[rd]\ar@{-}[llld]&&\\
&\ttroisun\tun\ar@{-}[ld]\ar@{-}[rd]|!{[ld];[rr]}\hole&&\ptroisun\tun \ar@{-}[llld]\ar@{-}[rrrd]|!{[rd];[rr]}\hole
&&\tdeux\tdeux \ar@{-}[ld]&&\tun\ttroisun\ar@{-}[ld] \ar@{-}[rrrd]|!{[rd];[rr]}\hole&&\tun\ptroisun \ar@{-}[ld] \ar@{-}[rd]&\\
\ttroisdeux\tun\ar@{-}[rd]&&\tquatreun\ar@{-}[ld] \ar@{-}[rd]&&\pquatrecinq \ar@{-}[ld] \ar@{-}[rrrd]|!{[rd];[rr]}\hole
&&\pquatresix\ar@{-}[ld]&&\pquatreun\ar@{-}[ld] \ar@{-}[rd]&&\tun\ttroisdeux \ar@{-}[ld]\\
&\tquatretrois \ar@{-}[rd]&&\tquatredeux\ar@{-}[rrd]|!{[ld];[rr]}\hole&&\pquatresept\ar@{-}[llld]\ar@{-}[rrrd]&&
\pquatredeux\ar@{-}[lld]|!{[ll];[rd]}\hole&&\pquatretrois \ar@{-}[ld]&\\
&&\tquatrequatre\ar@{-}[rrrd]&&&\pquatrehuit\ar@{-}[d]&&&\pquatrequatre\ar@{-}[llld]&&\\
&&&&&\tquatrecinq&&&&&}$$

\subsection{Isomorphism with the weak Bruhat order on permutations}

\begin{lemma} \label{11}
Let $\sigma \in \S_n$ and let $P=\Psi_n(\sigma)$. Let $1\leq i<j \leq n$. Then $\sigma$ is of the form $(\ldots ij \ldots)$ if, and only if,
the three following conditions are satisfied:
\begin{itemize}
\item $i<_h j$ in $P$.
\item If $x<_h j$ in $P$, then $x\leq_h i$ or $x \geq_r i$.
\item If $x>_h i$ in $P$, then $x \geq_h j$ or $x \leq_r j$.
\end{itemize}\end{lemma}

\begin{proof} $\Longrightarrow$. By definition of $P$, indeed $i<_h j$ in $P$.

If $x<_h j$ in $P$, then $x<j$ and $\sigma^{-1}(x)<\sigma^{-1}(j)$.
If $x=i$, then $x \leq_h i$ and $x \geq_r i$. If $x\neq i$, then $x$ appears before $j$ in the word representing $\sigma$, so it appears before $i$.
So if $x<i$, then $x \leq_h i$ and if $x>i$, $x \geq_r i$.

If $x>_h i$, then $x>i$ and $\sigma^{-1}(x)> \sigma^{-1}(i)$. If $x=j$, then $x \geq_h j$ and $x \leq_r j$.
If $x\neq j$, then $x$ appears after $i$ in the word representing $\sigma$, so it appears after $j$. If $x >j$, then $x \geq_h j$ and
if $x <j$, then $x \leq_r j$. \\

$\Longleftarrow$. As $i<_h j$, $i$ appears before $j$ in the word representing $\sigma$. 
If there is a letter $x$ between $i$ and $j$ in this word, three cases are possible.
\begin{itemize}
\item $x<i<j$. Then $x <_h j$ and $x <_r i$, so we do not have $x \leq_h i$ nor $x \geq_r i$. This contradicts the second condition.
\item $i<x<j$. Then $x<_h j$ and $x>_h i$, so we do not have $x \leq_h i$ nor $x \geq_r i$. This contradicts the second condition.
\item $i<j<x$. Then $x>_h i$ and $x>_r j$, so we do not have $x \geq_h j$ nor $x \leq_r j$. This contradicts the third condition.
\end{itemize}
As a consequence, $i$ and $j$ are two consecutive letters of the word representing $\sigma$. \end{proof}\\

{\bf Notation}. Let $P$ be a double poset. We put $E(P)=\{(i,j) \in P^2\:\mid\: i<_h j\}$. \\

{\bf Remark.} By definition of the partial order on $\PP(n)$, $P \leq Q$ if, and only if, $E(Q)\subseteq E(P)$.
Consequently, $P=Q$ if, and only if, $E(P)=E(Q)$.

\begin{lemma}\label{12}
Let $1\leq i<j \leq n$ and $\sigma$ be a permutation of the form $(\ldots ij\ldots)$. We put $\tau=(ij) \circ \sigma=(\ldots ji \ldots)$
(the other letters being unchanged). Then $E(\Psi_n(\tau))=E(\Psi_n(\sigma))-\{(i,j)\}$.
\end{lemma}

\begin{proof} We put $P=\Psi_n(\sigma)$ and $Q=\Psi_n(\tau)$.

$E(Q)\subseteq E(P)-\{(i,j)\}$. If $(k,l) \in E(Q)$, then $k<l$ and $\tau^{-1}(k)<\tau^{-1}(l)$, so $(k,l) \neq (i,j)$.
If $k,l \neq i,j$, then $\tau^{-1}(k)=\sigma^{-1}(k)$ and $\tau^{-1}(l)=\sigma^{-1}(l)$, so $(k,l) \in E(P)$.
If $k=i$ or $j$, then $l$ appears in the word representing $\tau$ after $i$ or $j$, so after $i$ and $j$, so it also appears in the word representing $\sigma$
after $i$ and $j$. As a consequence, $(k,l) \in E(P)$. If $l=i$ or $j$, we can prove in the same way that $(k,l) \in E(P)$.

$E(P)-\{(i,j)\} \subseteq E(Q)$. Similar proof. \end{proof}\\

{\bf Notation.} Let $P,Q \in \PP(n)$. We assume that $P=Q=\{1,\ldots,n\}$ as totally ordered sets.
We shall say that $P \preceq Q$ if there exists $(i,j) \in E(P)$ such that $E(Q)=E(P)-\{(i,j)\}$. 

\begin{lemma}\label{13}
Let $P \leq Q$ in $\PP(n)$. There exists $P_0=P, \ldots,P_k=Q$, such that $P=P_0\preceq P_1\preceq \ldots \preceq P_k=Q$.
\end{lemma}

\begin{proof} By definition of the partial order of $\PP(n)$, if $i <_h j$ in $Q$, then $i<_h j$ in $P$. So $E(Q) \subseteq E(P)$.
We proceed by induction on $k=E(P)-E(Q)$. If $k=0$, then $P=Q$ and the result is obvious.
Let us assume that $k\geq 1$. We put $\sigma=\Psi_n^{-1}(P)$ and $\tau=\Psi_n^{-1}(Q)$. 
We choose $(i,j) \in E(P)-E(Q)$, such that the distance $d$ between the letters $i$ and $j$ in the word representing $\sigma$ is minimal.
Let us assume that $d \geq 2$. As $i<_h j$ in $P$, there exists a letter $x$ such that $\sigma=(\ldots i \ldots x\ldots j \ldots)$.
Three cases are possible.
\begin{itemize}
\item If $x<i<j$, then, in $P$, $x<_r i$, $x<_h j$ and $i<_h j$. Hence, $(x,i) \notin E(P)$, so $(x,i) \notin E(Q)$ and $x<_r i$ in $Q$ ;
as $(i,j) \notin E(Q)$, $i<_r j$ in $Q$ ; finally, $x<_r j$ in $Q$. Consequently, $(x,j) \in E(P)-E(Q)$: contradicts the minimality of $d$.
\item If $i<x<j$, then, in $P$, $i<_h x$, $x<_h j$ and $i<_h j$. By minimality of $d$, $i<_h x$ and $x<_h j$ in $Q$, so $i<_h j$ in $Q$:
contradicts that $(i,j) \in E(P)-E(Q)$.
\item If $i<j<x$, then, in $P$, $i<_h x$, $j<_r x$ and $i<_h j$. By minimality of $d$, $i<_h x$ in $Q$.
As $(j,x) \notin E(P),$ $(j,x) \notin E(Q)$, so $j<_r x$ in $Q$ ; as $(i,j) \notin E(Q)$, $i<_r j$ in $Q$, and finally $i<_r x$ in $Q$:
contradicts $i<_h x$ in $Q$.
\end{itemize}
We deduce that $d=1$: $i$ and $j$ are two consecutive letters in $\sigma$. We then take $P_1=\Psi_n^{-1}((ij) \circ \sigma)$.
By lemma \ref{12}, $E(P_1)=E(P)-\{(i,j)\}$. We then apply the induction hypothesis to the couple $(P_1,Q)$. \end{proof}

\begin{lemma} \label{14}
Let $P,Q \in \PP(n)$, such that $P \preceq Q$. We put $E(Q)=E(P)-\{(i,j)\}$. Then $i,j$ are two consecutive letters in $\Psi_n^{-1}(P)$.
Moreover, $\Psi_n^{-1}(Q)$ is obtained by permuting the two consecutive letters $ij$ in $\Psi_n^{-1}(P)$.
\end{lemma}

\begin{proof} Let us prove that $i,j$ satisfy the three conditions of lemma \ref{11}. As $(i,j) \in E(P)$, $i<_h j$ in $P$. If $x<_h j$ in $P$, three cases are possible.
\begin{itemize}
\item If $x=i$, then $x\leq_h i$ and $x\geq_r i$ in $P$.
\item If $x<i$, let us assume that $x<_r i$ in $P$. Hence, $(x,i) \notin E(P)$, so $(x,i) \notin E(Q)$ and $x<_r i$ in $Q$.
As $i<_r j$ in $Q$, $x<_r j$ in $Q$. As $E(P)=E(Q) \cup\{(i,j)\}$, $(x,j) \notin E(P)$ and $x<_r j$ in $P$: contradiction.
So $x<_h j$.
\item If $x>i$, let us assume that $x>_h i$ in $P$. As $E(Q)=E(P)-\{(i,j)\}$, $i<_h x$ and $x<_h j$ in $Q$, so $i<_h j$ in $Q$:
contradiction, $(i,j) \notin E(Q)$. So $x>_r i$.
\end{itemize}

Let us now prove the third condition. If $x>_h i$ in $P$, three cases are possible.
\begin{itemize}
\item If $x=j$, then $x\geq_h j$ and $x \leq_r j$ in $P$.
\item If $x<j$, let us assume that $x<_h j$ in $P$. As $E(Q)=E(P)-\{(i,j)\}$, $x<_h j$ and $i <_h x$ in $Q$,
so $i<_h j$ in $Q$: contradiction, $(i,j) \notin E(Q)$. So $x<_r j$ in $Q$.
\item If $x>j$, let us assume that $x>_r j$ in $P$. As $E(Q)=E(P)-\{(i,j)\}$, $x>_r j$ and $j>_r i$ in $Q$, so $x>_r i$ in $Q$.
As $E(P)=E(Q)\cup\{(i,j)\}$, $x>_r i$ in $P$: contradicts $x>_h i$ in $P$. So $x>_h j$ in $P$.
\end{itemize}
Finally, the three conditions of lemma \ref{11} are satisfied. \\

We put $\Psi_n^{-1}(P)=\sigma$ and $\Psi_n^{-1}(Q)=\tau$. By lemma \ref{11}, $ij$ are two consecutive letters in the word representing $\sigma$.
Let $\varsigma$ be the permutation obtained by permuting these two letters. By lemma \ref{12}, $E(\Psi_n(\varsigma))=E(P)-\{(i,j)\}=E(Q)$,
so $\Psi_n(\varsigma)=Q$ and $\tau=\varsigma$. \end{proof}

\begin{theo}
We partially order $\S_n$ by the weak Bruhat order \cite{Stanley1,Stanley2}. For all $n \geq 0$, the bijection $\Psi_n$ is an isomorphism of posets, that is to say:
for all $\sigma,\tau \in \S_n$, $\sigma \leq \tau$ if, and only if, $\Psi_n(\sigma)\leq \Psi_n(\tau)$ in $\PP(n)$.
\end{theo}

\begin{proof} We consider $\sigma,\tau \in \S_n$. We put $\Psi_n(\sigma)=P$ and $\Psi_n(\tau)=Q$. \\

Let us assume that $\sigma \leq\tau$ in $\S_n$. By definition of the weak Bruhat order, there exists 
$\sigma_0=\sigma,\ldots,\sigma_k=\tau$, such that $\sigma_{p+1}$ is obtained from $\sigma_p$ by permuting two consecutives letters $ij$, with $i<j$,
in the word representing $\sigma_p$. By lemma \ref{12}, for all $p$, $E(\Psi_n(\sigma_{p+1}))\subseteq E(\Psi_n(\sigma_p))$.
Consequently, $E(Q) \subseteq E(P)$, so $P\leq Q$. \\

Let us assume that $P \leq Q$. From lemma \ref{13}, there exists $P=P_0\preceq \ldots \preceq P_k=Q$.
We put $\sigma_p=\Psi_n^{-1}(P_p)$ for all $0\leq p \leq k$. From lemma \ref{14}, we obtain $\sigma_{p+1}$ from $\sigma_p$ 
by permuting two consecutives letters $ij$, with $i<j$, in the word representing $\sigma_p$. By definition of the weak Bruhat order, $\sigma \leq \tau$. \end{proof}

\subsection{Properties of the partial order}

\begin{prop}\label{16}
Let $P_1,Q_1 \in \PP(k)$, $P_2,Q_2 \in \PP(l)$. The following conditions are equivalent:
\begin{enumerate}
\item $P_1 P_2 \leq Q_1 Q_2$.
\item $P_1 \prodh P_2 \leq Q_2 \prodh Q_2$.
\item $P_1 \leq Q_1$ and $P_2 \leq Q_2$.
\end{enumerate}\end{prop}

\begin{proof}
We put $\theta_i=\theta_{P_i,Q_i}$, $i=1,2$. It is clear that the unique increasing bijection from $P_1 P_2$ to
$Q_1Q_2$ and the unique increasing bijection from $P_1\prodh P_2$ to $Q_1 \prodh Q_2$ are both $\theta=\theta_1 \otimes \theta_2$. 
As an immediate consequence, $1$ or $2$ implies $3$.\\

$3\Longrightarrow 1$. Let us assume that $\theta(i)<_h \theta(j)$ in $Q_1Q_2$. Two cases are possible.
\begin{itemize}
\item $i,j \in P_1$. Then $\theta_1(i)<_h \theta_1(j)$ in $Q_1$, so $i<_h j$ in $P_1$, so $i<_h j$ in $P_1P_2$.
\item $i,j \in P_2$. Then $\theta_2(i)<_h \theta_2(j)$ in $Q_2$, so $i<_h j$ in $P_2$, so $i<_h j$ in $P_1P_2$.
\end{itemize}
So $P_1P_2 \leq Q_1Q_2$. \\

$3 \Longrightarrow 2$. Let us assume that $\theta(i)<_h \theta(j)$ in $Q_1 \prodh Q_2$. Three cases are possible.
\begin{itemize}
\item $i,j \in P_1$. Then $\theta_1(i)<_h \theta_1(j)$ in $Q_1$, so $i<_h j$ in $P_1$, so $i<_h j$ in $P_1 \prodh P_2$.
\item $i,j \in P_2$. Then $\theta_2(i)<_h \theta_2(j)$ in $Q_2$, so $i<_h j$ in $P_2$, so $i<_h j$ in $P_1 \prodh P_2$.
\item $i \in P_1$ and $j \in P_2$. Then $i<_h j$ in $P_1 \prodh P_2$.
\end{itemize}
So $P_1 \prodh P_2 \leq Q_1 \prodh Q_2$.\end{proof}

\begin{defi}
Let $P=(P,\leq_h,\leq_r)$ be a plane poset. We put $\iota(P)=(P,\leq_r,\leq_h)$. Note that $\iota$ is an involution of $\PP$.
\end{defi}

\begin{prop}
For any $P,Q \in \PP(n)$, $P\leq Q$ if, and only if, $\iota(Q)\leq \iota(P)$.
\end{prop}

\begin{proof} Let $P,Q \in \PP(n)$. As the total orders on $R$ and $\iota(R)$ are identical for any $R\in \PP(n)$,  
the unique increasing bijection from $\iota(Q)$ to $\iota(P)$ is $\theta_{Q,P}$. Hence:
\begin{eqnarray*}
P\leq Q&\Longleftrightarrow&\forall x,y \in P, \:(\theta_{P,Q}(x) \leq_h \theta_{P,Q}(y)\mbox{ in }Q) \Longrightarrow (x \leq_h y \mbox{ in }P)\\
&\Longleftrightarrow&\forall x,y \in \iota(P), \:(\theta_{P,Q}(x) \leq_r \theta_{P,Q}(y)\mbox{ in }\iota(Q)) \Longrightarrow 
(x \leq_r y \mbox{ in }\iota(P))\\
&\Longleftrightarrow&\forall x',y' \in \iota(Q), \:(x' \leq_r y' \mbox{ in }\iota(Q)) \Longrightarrow 
(\theta_{Q,P}(x') \leq_r \theta_{Q,P}(y') \mbox{ in }\iota(P))\\
&\Longleftrightarrow&\iota(Q)\leq \iota(P).
\end{eqnarray*}\end{proof}\\

%
%

{\bf Remark.} Let $P \in \PP(n)$. We put $\sigma=\Psi_n^{-1}(P)$. Then $\Psi_n^{-1}(\iota(P))= \sigma \circ (n\ldots 1)$.

\subsection{Restriction to plane forests}

Let us consider the restriction the partial order to the set of plane forests. 
\begin{defi}\label{19}
Let $F$ be a plane forest and let $s$ be a vertex of $F$ which is not a leaf. 
The \emph{transformation of $F$ at vertex $s$} is the plane forest obtained in one of the following way:
$$\begin{array}{rccc}
\mbox{If $s$ is not a root: }&
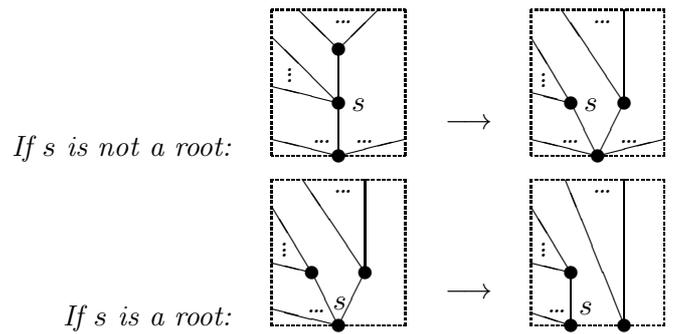
\begin{picture}(55,60)(-25,0)
\put(0,0){\circle*{5}}\put(0,0){\line(0,1){20}}\put(0,20){\circle*{5}}\put(0,20){\line(0,1){20}}\put(0,40){\circle*{5}}
\put(0,40){\line(-1,1){15}}\put(0,40){\line(1,1){15}}\put(5,17){$s$}\put(-2,50){.}\put(0,50){.}\put(2,50){.}
\put(0,0){\line(-4,1){25}}\put(0,0){\line(4,1){25}}\put(-10,5){.}\put(-8,5){.}\put(-6,5){.}\put(10,5){.}\put(8,5){.}\put(6,5){.}
\put(0,20){\line(-4,1){25}}\put(0,20){\line(-1,1){25}}\put(-20,28){.}\put(-20,30){.}\put(-20,32){.}\put(-25,0){\dashbox{1}(50,55)}
\end{picture}&\begin{picture}(22,0)(0,0)\put(0,10){$ \longrightarrow $}
\end{picture}&\begin{picture}(55,60)(-25,0)
\put(0,0){\circle*{5}}\put(0,0){\line(1,2){10}}\put(0,0){\line(-1,2){10}}\put(-10,20){\circle*{5}}\put(10,20){\circle*{5}}
\put(-5,17){$s$}\put(0,0){\line(-4,1){25}}\put(0,0){\line(4,1){25}}\put(-14,5){.}\put(-12,5){.}\put(-10,5){.}\put(12,5){.}
\put(10,5){.}\put(8,5){.}\put(10,20){\line(0,1){35}}\put(10,20){\line(-2,3){23.5}}\put(-2,50){.}\put(0,50){.}\put(2,50){.}
\put(-10,20){\line(-3,5){15}}\put(-10,20){\line(-4,1){15}}\put(-22,26){.}\put(-22,28){.}\put(-22,30){.}\put(-25,0){\dashbox{1}(50,55)}
\end{picture},\\
\mbox{If $s$ is a root: }&\begin{picture}(55,60)(-25,0)
\put(0,0){\circle*{5}}\put(0,0){\line(1,2){10}}\put(0,0){\line(-1,2){10}}\put(-10,20){\circle*{5}}\put(10,20){\circle*{5}}
\put(-2,6){$s$}\put(0,0){\line(-4,1){25}}\put(-12,5){.}\put(-10,5){.}\put(-8,5){.}\put(10,20){\line(0,1){35}}
\put(10,20){\line(-2,3){23.5}}\put(-2,50){.}\put(0,50){.}\put(2,50){.}\put(-10,20){\line(-3,5){15}}\put(-10,20){\line(-4,1){15}}
\put(-22,26){.}\put(-22,28){.}\put(-22,30){.}\put(-25,0){\dashbox{1}(50,55)}
\end{picture}&\begin{picture}(22,0)(0,0)
\put(0,10){$ \longrightarrow $}
\end{picture}&\begin{picture}(55,60)(-25,0)
\put(-10,0){\circle*{5}}\put(-10,20){\circle*{5}}\put(10,0){\circle*{5}}\put(-7,4){$s$}\put(-10,0){\line(0,1){20}}
\put(-10,20){\line(-3,5){15}}\put(-10,20){\line(-4,1){15}}\put(-22,26){.}\put(-22,28){.}\put(-22,30){.}\put(10,0){\line(0,1){55}}
\put(10,0){\line(-2,5){22}}\put(2,50){.}\put(0,50){.}\put(-2,50){.}\put(-10,0){\line(-3,1){15}}\put(-25,0){\dashbox{1}(50,55)}
\put(-15,5){.}\put(-17,5){.}\put(-19,5){.}\end{picture},\end{array}$$
the part of the forest outside the frame being unchanged. 
\end{defi}

{\bf Remark.} Up to a vertical symmetry, these transformation are used in \cite{Foissy3} in order to define a partial order on
the set of plane forests, making it isomorphic to the Tamari poset.

\begin{prop}
Let $F,G$ be two plane forests of degree $n$. Then $F \leq G$ if, and only if, $G$ is obtained from $F$ by a finite number of transformations of definition \ref{19}.
\end{prop}

\begin{proof}
$\Longrightarrow$. By induction on $n$. If $n=1$, then $F=G=\tun$ and the result is obvious. Let us assume that $n \geq 2$. 

{\it First case}. Let us assume that $F$ is a plane tree $F=\tun \prodh F'$. We put $G=(\tun \prodh G')G''$, where $F',G',G''$ are plane forests.
The increasing bijection from $F'$ to $G'G''$ is $(\Theta_{F,G})_{\mid F'}$, as the root of $F$ and the root of $\tun \prodh F'$ are the smallest elements 
for their total orders. By the induction hypothesis, one can obtain $G'G''$ from $F'$ by a finite number of transformations. 
So $\tun \prodh(G'G'')$ can be obtained from $F$ by a finite number of transformations. Applying transformations to the left root of $\tun\prodh(G'G''))$,
one can obtain $G$ from$\tun \prodh(G'G'')$ by a finite number of transformations.

{\it Second case}. Let us assume that $F$ is not a tree. We put $F=(\tun \prodh F')F''$, where $F',F''$ are plane forests, $F''\neq 1$.
We put $\Theta_{F,G}(\tun \prodh F')=G'$ and $\Theta_{F,G}(F'')=G''$. Let us consider $x \in G'$ and $y\in G''$. 
Then $\Theta_{F,G}^{-1}(x)\leq_r \Theta_{F,G}^{-1}(y)$ in $F$,
so $x\leq y$ in $G$. Moreover, if $x<_h y$ in $G$, then $(x,y) \in E(G)$. As $F\leq G$, $(\Theta_{F,G}^{-1}(x),
\Theta_{F,G}^{-1}(y)) \in E(F)$, so $\Theta_{F,G}^{-1}(x)<_h \Theta_{F,G}^{-1}(y)$ in $F$: contradiction. So $x<_r y$ in $G$. 
Consequently, $G=G'G''$. By proposition \ref{16}, $\tun \prodh F' \leq G'$ and $F'' \leq G''$. By the induction hypothesis, one can obtain $G'$ from $\tun \prodh F'$
and $G''$ from $F''$ by a certain number of transformations. Hence, we can obtain $G=G'G''$ from $F=(\tun \prodh F')F''$ by a certain number of transformations.\\

$\Longleftarrow$. By transitivity, It is enough to prove it if $G$ is obtained from $F$ by a transformation. It is clear that this transformation does not affect
the total order on the vertices of $F$, so the increasing bijection from $F$ to $G$ is the identity. Obviously, if $x \leq_h y$ in $G$,
then $x \leq_h y$ in $F$, so $F \leq G$.\end{proof}\\

From \cite{Foissy3}, we recover the injection fo the Tamari poset into the Bruhat poset:

\begin{cor}
The poset of plane forests of degree $n$ is isomorphic to the Tamari poset on plane binary trees with $n+1$ leaves (or $n$ internal vertices).
\end{cor}

\section{Link with the infinitesimal structure}

\subsection{A lemma on the Bruhat order}

\begin{lemma}\label{22}
Let $P,Q,R$ be three double posets.
\begin{enumerate}
\item $P\prodh Q \leq R$ if, and only if,  there exists a biideal $I_0$ of $R$, such that $P\leq R\setminus I_0$
and $Q\leq I_0$. Moreover, if this holds, $I_0$ is unique and $I_0=\theta_{PQ,R}(Q)$.
\item $PQ \leq R$ if, and only if, there exists a plane subposet $I_0$ of $R$, such that $R=(R\setminus I_0) I_0$,
$P \leq R\setminus I_0$  and $Q\leq I_0$. Moreover, if this holds, $I_0$ is unique and $I_0=\theta_{P\prodh Q,R}(Q)$.
\item $\iota(PQ) \leq R$ if, and only if, there exists a biideal $I_0$ of $R$, such that $\iota(R\setminus I_0)\leq P$
and $\iota(I_0)\leq Q$. Moreover, if this holds, $I_0$ is unique and $I_0=\theta_{PQ,R}(Q)$.
\item $\iota(P\prodh Q) \leq R$ if, and only if, there exists a plane subposet $I_0$ of $R$, such that $R=(R\setminus I_0) I_0$,
$\iota(R\setminus I_0)\leq P$  and $\iota(I_0)\leq Q$. Moreover, if this holds, $I_0$ is unique and $I_0=\theta_{P\prodh Q,R}(Q)$.
\end{enumerate}\end{lemma}

\begin{proof}  1. $\Longrightarrow$. We consider $I_0=\theta_{P\prodh Q,R}(Q)$. If $x \in I_0$ and $y\in R$ satisfy $x \leq y$,
then $\theta_{P\prodh Q,R}^{-1}(x) \in Q$ and $\theta^{-1}_{P\prodh Q,R}(x) \leq \theta^{-1}_{P\prodh Q,R}(y)$, 
so $\theta^{-1}_{P\prodh Q,R}(y)\in Q$ and $y\in I_0$: hence, $I_0$ is a biideal. Moreover, $\theta_{P,R\setminus I_0}$ is the restriction 
of $\Theta_{P\prodh Q,R}$ to $P$ and $\theta_{Q,I_0}$ is the restriction of $\Theta_{P\prodh Q,R}$ to $Q$; as $P\prodh Q \leq R$,
$P\leq R\setminus I_0$ and $Q \leq I_0$.

1. $\Longleftarrow$.  Let $I$ be a such a biideal. As $\iota(Q) \leq I$, $|I|=|Q|=|I_0|$. As $I$ is a biideal, $|I|$ is made of the $|Q|$ greatest elements of $R$.
As $\theta_{P\prodh Q,R}$ is increasing and as the $|Q|$ greatest elements of $P\prodh Q$ are the elements of $Q$, $I=\theta_{P\prodh Q,R}(Q)=I_0$.
Let $x,y\in P\prodh Q$, such that $\theta_{P\prodh Q,R}(x) \leq_h \theta_{P\prodh Q,R}(y)$. As $I=I_0$ is a biideal, three cases are possible:
\begin{itemize}
\item $x,y \in P$. As $P\leq R\setminus I_0$, then $x\leq_h y$ in $P$, hence in $P\prodh Q$.
\item $x,y \in Q$. As $Q\leq I_0$, then $x\leq_h y$ in $Q$, hence in $P\prodh Q$.
\item $x\in P$, $y \in Q$. Then $x \leq_h y$ in $P\prodh Q$.
\end{itemize}
So $P\prodh Q \leq R$. \\

2. $\Longrightarrow$.  We consider $I_0=\theta_{PQ,R}(Q)$. 
If $x\in R\setminus I_0$ and $y\in I_0$, then  $\theta_{PQ,R}^{-1}(x) \in P$ and  $\theta_{PQ,R}^{-1}(y) \in Q$,
so $\theta_{PQ,R}^{-1}(x)\leq_r \theta_{PQ,R}^{-1}(y)$. As $PQ\leq R$, $x\leq_r y$, so $R=(R\setminus I_0)I_0$.
Moreover, $\theta_{P,R\setminus I_0}$ is the restriction of $\Theta_{PQ,R}$ to $P$ and $\theta_{Q,I_0}$ is the restriction of $\Theta_{PQ,R}$ to $Q$; 
as $PQ \leq R$, $P\leq R\setminus I_0$ and $Q \leq I_0$.\\

2. $\Longleftarrow$. Let $I$ be such a subposet. As $Q \leq I$, $|Q|=|I|=|I_0|$. Moreover, $R=(R\setminus I) I$, so the elements of $I$
are the $|Q|$ greatest elements of $R$. Hence, $I=I_0=\theta_{PQ,R}(Q)$. 
Let $x,y\in PQ$, such that $\theta_{PQ,R}(x) \leq_h \theta_{PQ,R}(y)$. As $R=(R\setminus I)I$, $x,y$ are both in $P$ or both in $Q$.
As $P\leq R\setminus I$ and $Q \leq I$, $x \leq_h y$ in $P$ or in $Q$, hence in $PQ$. So $PQ \leq R$. \\

3 and 4. Reformulations of the first two points, with the observations that $\iota(PQ)=\iota(P)\prodh \iota(Q)$, $\iota(P\prodh Q)=\iota(P)\iota(Q)$
and $\iota$ is decreasing for the Bruhat order $\leq$. \end{proof}

\subsection{Construction of the Hopf pairing}

{\bf Notations.} Let $P,Q \in \PP(n)$. We put:
\begin{eqnarray*}
\phi(P,Q)&=&\sharp\{(x,y)\in P^2\mid x<_r y\mbox{ and }\theta_{P,Q}(x)<_h \theta_{P,Q}(y)\}\\
&&+\sharp\{(x,y)\in P^2\mid x<_h y\mbox{ and }\theta_{P,Q}(x)<_r \theta_{P,Q}(y)\}.
\end{eqnarray*}

\begin{theo}\label{23}
Let $P,Q\in \PP$. We put:
$$\langle P,Q \rangle_q=\left\{\begin{array}{l}
q^{\phi(P,Q)} \mbox{ if }\iota(P)\leq Q,\\
0\mbox{ if not.}
\end{array}\right.$$
This pairing is bilinearly extended to $\h_{\PP}$. 
Then $\langle-,-\rangle_q$ is a symmetric Hopf pairing on $(\h_{\PP},m,\Delta_q)$. It is nondegenerate if, and only if, $q\neq 0$.
\end{theo}

{\bf Examples.} 
$$\begin{array}{c|c}
&\tun\\
\hline \tun&1
\end{array}\hspace{1cm}\begin{array}{c|c|c}
&\tdeux&\tun\tun\\
\hline \tdeux&0&q\\
\hline \tun\tun&q&1
\end{array}\hspace{1cm}\begin{array}{c|c|c|c|c|c|c}
&\ttroisdeux&\ttroisun&\ptroisun&\tdeux\tun&\tun\tdeux&\tun\tun\tun\\
\hline\ttroisdeux&0&0&0&0&0&q^3\\
\hline\ttroisun&0&0&0&0&q^3&q^2\\
\hline\ptroisun&0&0&0&q^3&0&q^2\\
\hline \tdeux\tun&0&0&q^3&0&q^2&q\\
\hline \tun \tdeux&0&q^3&0&q^2&0&q\\
\hline\tun\tun\tun&q^3&q^2&q^2&q&q&1
\end{array}$$

\begin{proof} Let $P,Q \in \PP$. If $\iota(P)\leq Q$, then $\iota(Q)\leq \iota^2(P)=P$. Moreover, as $\theta_{P,Q}$ is bijective,
of inverse $\theta_{Q,P}$:
\begin{eqnarray*}
&&\sharp\{(x,y)\in P^2\mid x<_r y\mbox{ and }\theta_{P,Q}(x)<_h \theta_{P,Q}(y)\}\\
&=&\sharp\{(x',y')\in Q^2\mid \theta_{Q,P}(x')<_r \theta_{Q,P}(y')\mbox{ and }x'\leq_h y'\};\\ \\
&&\sharp\{(x,y)\in P^2\mid x<_h y\mbox{ and }\theta_{P,Q}(x)<_r \theta_{P,Q}(y)\}\\
&=&\sharp\{(x',y')\in Q^2\mid \theta_{Q,P}(x')<_h \theta_{Q,P}(y')\mbox{ and }x'\leq_r y'\}.
\end{eqnarray*}
So $\phi(P,Q)=\phi(Q,P)$, and this pairing is symmetric. \\

Let $P,Q,R \in \PP$. Let us prove that $\langle PQ,R\rangle_q=\langle P\otimes Q,\Delta_q(R)\rangle_q$. 

\emph{First case}. Let us assume that $\iota(PQ)\leq R$. By lemma \ref{22}, there exists a unique biideal $I_0$ of $R$, such that 
$\iota(P)\leq R\setminus I_0$ and $\iota(Q)\leq I_0$. Hence:
\begin{eqnarray*}
\langle P\otimes Q,\Delta_q(R)\rangle_q
&=&\sum_{\mbox{\scriptsize $I$ biideal of $R$}}q^{h_{R\setminus I}^I}\langle P,R\setminus I\rangle_q \langle Q,I\rangle_q\\
&=&q^{h_{R\setminus I_0}^{I_0}}\langle P,R\setminus I_0\rangle_q \langle Q,I_0\rangle_q+0\\
&=&q^{h_{R\setminus I_0}^{I_0}+\phi(P,R\setminus I_0)+\phi(Q,I_0)}.
\end{eqnarray*}
Moreover:
\begin{eqnarray*}
\Phi(PQ,R)&=&\sharp\{(x,y)\in P^2\mid x<_r y\mbox{ and } \phi_{PQ,R}(x) <_h \phi_{PQ,R}(y)\}\\
&&+\sharp\{(x,y)\in P^2\mid x<_h y\mbox{ and } \phi_{PQ,R}(x) <_r \phi_{PQ,R}(y)\}\\
&&+\sharp\{(x,y)\in Q^2\mid x<_r y\mbox{ and } \phi_{PQ,R}(x) <_h \phi_{PQ,R}(y)\}\\
&&+\sharp\{(x,y)\in Q^2\mid x<_h y\mbox{ and } \phi_{PQ,R}(x) <_r \phi_{PQ,R}(y)\}\\
&&+\sharp\{(x,y)\in P\times Q\mid x<_r y\mbox{ and } \phi_{PQ,R}(x) <_h \phi_{PQ,R}(y)\}\\
&&+\sharp\{(x,y)\in P\times Q\mid x<_h y\mbox{ and } \phi_{PQ,R}(x) <_r \phi_{PQ,R}(y)\}\\
&&+\sharp\{(x,y)\in Q\times P\mid x<_r y\mbox{ and } \phi_{PQ,R}(x) <_h \phi_{PQ,R}(y)\}\\
&&+\sharp\{(x,y)\in Q\times P\mid x<_h y\mbox{ and } \phi_{PQ,R}(x) <_r \phi_{PQ,R}(y)\}.
\end{eqnarray*}
As $I_0=\theta_{PQ,R}(Q)$, $\theta_{P,R\setminus I_0}$ is the restriction to $P$ of $\theta_{P\prodh Q,R}$ and $\theta_{Q,I_0}$ 
is the restriction to $Q$ of $\theta_{P\prodh Q,R}$, so:
\begin{eqnarray*}
\phi(P,R\setminus I_0)&=&\sharp\{(x,y)\in P^2\mid x<_r y\mbox{ and } \phi_{PQ,R}(x) <_h \phi_{PQ,R}(y)\}\\
&&+\sharp\{(x,y)\in P^2\mid x<_h y\mbox{ and } \phi_{P,R\setminus I_0}(x) <_r \phi_{P,R\setminus I_0}(y)\},\\ \\
\phi(Q,I_0)&=&\sharp\{(x,y)\in Q^2\mid x<_r y\mbox{ and } \phi_{PQ,R}(x) <_h \phi_{PQ,R}(y)\}\\
&&+\sharp\{(x,y)\in Q^2\mid x<_h y\mbox{ and } \phi_{P,R\setminus I_0}(x) <_r \phi_{P,R\setminus I_0}(y)\}.
\end{eqnarray*}
If $x\in P$ and $y\in Q$, then $x<_r y$. So:
\begin{eqnarray*}
0&=&\sharp\{(x,y)\in P\times Q\mid x<_h y\mbox{ and } \phi_{PQ,R}(x) <_r \phi_{PQ,R}(y)\}\\
&=&\sharp\{(x,y)\in Q\times P\mid x<_r y\mbox{ and } \phi_{PQ,R}(x) <_h \phi_{PQ,R}(y)\}\\
&=&\sharp\{(x,y)\in Q\times P\mid x<_h y\mbox{ and } \phi_{PQ,R}(x) <_r \phi_{PQ,R}(y)\},
\end{eqnarray*}
and:
\begin{eqnarray*}
h_{R\setminus I_0}^{I_0}&=&\sharp\{(x,y) \in P\times Q\mid  \phi_{PQ,R}(x) <_h \phi_{PQ,R}(y)\}\\
&=&\sharp\{(x,y) \in P\times Q\mid x<_r y\mbox{ and } \phi_{PQ,R}(x) <_h \phi_{PQ,R}(y)\}\\
\end{eqnarray*}
Finally, $\phi(PQ,R)=\phi(P,R\setminus I_0)+\phi(Q,I_0)+h_{R\setminus I_0}^{I_0}+0$. Hence:
$$\langle P\otimes Q,\Delta_q(R) \rangle_q=q^{\phi(PQ,R)}=\langle PQ,R\rangle_q.$$

\emph{Second case}. Let us assume that we do not have $\iota(PQ)\leq R$. By lemma \ref{22}, for any biideal $I$ of $R$, 
$\langle P,R\setminus I\rangle_q\langle Q,I\rangle_q=0$, So $\langle P\otimes Q,\Delta_q(R) \rangle_q=\langle PQ,R\rangle_q=0$. \\

Let us now study the degeneracy of this pairing. If $q=0$, the examples below show that $\tdeux$ is in the orthogonal of the pairing $\langle-,-\rangle_0$,
so this pairing is degenerate. Let us assume that $q\neq 0$. 

\emph{First step}. Let $P\in \PP(n)$. By definition of the pairing, $\langle \iota(P),P\rangle_q=q^{\phi(P,\iota(P))}$.
Moreover, as $\theta_{P,\iota(P)}=Id_P$:
\begin{eqnarray*}
\phi(P,\iota(P))&=&\sharp\{(x,y)\in P^2 \mid x<_r y\mbox{ and } x<_r y\}+\sharp\{(x,y)\in P^2 \mid x<_h y\mbox{ and } x<_h y\}\\
&=&\sharp\{(x,y)\in P^2 \mid x<_r y\}+\sharp\{(x,y)\in P^2 \mid x<_h y\}\\
&=&\sharp\{(x,y)\in P^2 \mid x< y\}\\
&=&\frac{n(n-1)}{2}.
\end{eqnarray*}
So $\langle P,\iota(P)\rangle_q=q^{\frac{n(n-1)}{2}}\neq 0$. \\
\emph{Second step}. 
Let us fix $n \geq 0$. We index the elements of $\PP(n)$ in such a way that if $\iota(P_i)<\iota(P_j)$ 
for the Bruhat order, then $i<j$. Let $x \in \PP(n)$, nonzero.  Let $i$ be the smallest integer such that $P_i$ appears in $x$.
Let $a$ be the coefficient of $P_i$ in $x$. If $j>i$, then it is not possible to have $\iota(P_j)\leq \iota(P_i)$, so $\langle P_j,\iota(P_i)\rangle_q=0$.
Consequently:
$$\langle x,\iota(P_i)\rangle_q=\langle a P_i,\iota(P_i)\rangle_q+0=aq^{\frac{n(n-1)}{2}}\neq 0.$$
So $x$ is not in the orthogonal of $\h_{\PP}$: the pairing is nondegenerate. \end{proof}\\

{\bf Remark.} This pairing is the pairing $\langle-,-\rangle_{q,0,1,0}$ of \cite{Foissy1}. 

\begin{prop}
We define a coproduct $\Delta'_q$ on $\h_\PP$ in the following way: for all $P\in \PP$,
$$\Delta'_q(P)=\sum_{P_1P_2=P}q^{|P_1||P_2|}P_1 \otimes P_2.$$
Then for all $x,y,z \in \h_\PP$, $\langle x\prodh y,z\rangle_q=\langle x\otimes y,\Delta'_q(z)\rangle_q$.
\end{prop}

\begin{proof} Let $P,Q,R \in \PP$. Let us prove that $\langle P\prodh Q,R\rangle_q=\langle P\otimes Q,\Delta'_q(R)\rangle_q$. 

\emph{First case}. Let us assume that $\iota(P\prodh Q)\leq R$. By lemma \ref{22}, there exists a unique $I_0\subseteq R$, such that 
$R=(R\setminus I_0)I_0$, $\iota(P)\leq R\setminus I_0$ and $\iota(Q)\leq I_0$. Hence:
\begin{eqnarray*}
\langle P\otimes Q,\Delta'_q(R)\rangle_q
&=&\sum_{R=R_1R_2}q^{|R_1||R_2|}\langle P,R\setminus I\rangle_q \langle Q,I\rangle_q\\
&=&q^{|R\setminus I_0||I_0|}\langle P,R\setminus I_0\rangle_q \langle Q,I_0\rangle_q+0\\
&=&q^{|R\setminus I_0||I_0|+\phi(P,R\setminus I_0)+\phi(Q,I_0)}.
\end{eqnarray*}
Moreover:
\begin{eqnarray*}
\Phi(P\prodh Q,R)&=&\sharp\{(x,y)\in P^2\mid x<_r y\mbox{ and } \phi_{P\prodh Q,R}(x) <_h \phi_{P\prodh Q,R}(y)\}\\
&&+\sharp\{(x,y)\in P^2\mid x<_h y\mbox{ and } \phi_{P\prodh Q,R}(x) <_r \phi_{P\prodh Q,R}(y)\}\\
&&+\sharp\{(x,y)\in Q^2\mid x<_r y\mbox{ and } \phi_{P\prodh Q,R}(x) <_h \phi_{P\prodh Q,R}(y)\}\\
&&+\sharp\{(x,y)\in Q^2\mid x<_h y\mbox{ and } \phi_{P\prodh Q,R}(x) <_r \phi_{P\prodh Q,R}(y)\}\\
&&+\sharp\{(x,y)\in P\times Q\mid x<_r y\mbox{ and } \phi_{P\prodh Q,R}(x) <_h \phi_{P\prodh Q,R}(y)\}\\
&&+\sharp\{(x,y)\in P\times Q\mid x<_h y\mbox{ and } \phi_{P\prodh Q,R}(x) <_r \phi_{P\prodh Q,R}(y)\}\\
&&+\sharp\{(x,y)\in Q\times P\mid x<_r y\mbox{ and } \phi_{P\prodh Q,R}(x) <_h \phi_{P\prodh Q,R}(y)\}\\
&&+\sharp\{(x,y)\in Q\times P\mid x<_h y\mbox{ and } \phi_{P\prodh Q,R}(x) <_r \phi_{P\prodh Q,R}(y)\}.
\end{eqnarray*}

As $I_0=\theta_{P\prodh Q,R}(Q)$, $\theta_{P,R\setminus I_0}$ is the restriction to $P$ of $\theta_{P\prodh Q,R}$ and $\theta_{Q,I_0}$ 
is the restriction to $Q$ of $\theta_{P\prodh Q,R}$, so:
\begin{eqnarray*}
\phi(P,R\setminus I_0)&=&\sharp\{(x,y)\in P^2\mid x<_r y\mbox{ and } \phi_{P\prodh Q,R}(x) <_h \phi_{P\prodh Q,R}(y)\}\\
&&+\sharp\{(x,y)\in P^2\mid x<_h y\mbox{ and } \phi_{P,R\setminus I_0}(x) <_r \phi_{P,R\setminus I_0}(y)\},\\ \\
\phi(Q,I_0)&=&\sharp\{(x,y)\in Q^2\mid x<_r y\mbox{ and } \phi_{P\prodh Q,R}(x) <_h \phi_{P\prodh Q,R}(y)\}\\
&&+\sharp\{(x,y)\in Q^2\mid x<_h y\mbox{ and } \phi_{P,R\setminus I_0}(x) <_r \phi_{P,R\setminus I_0}(y)\}.
\end{eqnarray*}
If $x\in P$ and $y\in Q$, then $x<_h y$. So:
\begin{eqnarray*}
0&=&\sharp\{(x,y)\in P\times Q\mid x<_r y\mbox{ and } \phi_{P\prodh Q,R}(x) <_r \phi_{P\prodh Q,R}(y)\}\\
&=&\sharp\{(x,y)\in Q\times P\mid x<_r y\mbox{ and } \phi_{P\prodh Q,R}(x) <_h \phi_{P\prodh Q,R}(y)\}\\
&=&\sharp\{(x,y)\in Q\times P\mid x<_h y\mbox{ and } \phi_{P,R\setminus I_0}(x) <_r \phi_{P,R\setminus I_0}(y)\}.
\end{eqnarray*}
Moreover:
\begin{eqnarray*}
|R\setminus I_0||I_0|&=&\sharp\{(x',y') \in (R\setminus I_0)\times I_0\mid x'<_r y'\}\\
&=&\sharp\{(x,y) \in P\times Q\mid \phi_{P\prodh Q,R}(x)<_r \phi_{P\prodh Q,R}(y)\}\\
&=&\sharp\{(x,y)\in P\times Q\mid x<_h y\mbox{ and } \phi_{P,R\setminus I_0}(x) <_r \phi_{P,R\setminus I_0}(y)\}.
\end{eqnarray*}
Finally, $\phi(P\prodh Q,R)=\phi(P,R\setminus I_0)+\phi(Q,I_0)+|R\setminus I_0||I_0|+0$. Hence:
$$\langle P\otimes Q,\Delta'_q(R) \rangle_q=q^{\phi(P\prodh Q,R)}=\langle P\prodh Q,R\rangle_q.$$

\emph{Second case}. Let us assume that we don't have $\iota(P\prodh Q)\leq R$. By lemma \ref{22}, if $R=(R\setminus I)I$, then 
$\langle P,R\setminus I\rangle_q\langle Q,I\rangle_q=0$, So $\langle P\otimes Q,\Delta'_q(R) \rangle_q=\langle P\prodh Q,R\rangle_q=0$. \end{proof}\\

{\bf Remark.} The coproduct $\Delta'_q$ is the coproduct $\Delta_{(0,0,q,0)}$ of \cite{Foissy1}. \\

Let us conclude this section by the case $q=0$.

\begin{prop}\begin{enumerate}
\item For all plane poset $P$:
$$\Delta_0(P)=\sum_{P_1P_2=P}P_1\otimes P_2.$$
\item For any plane posets $P,Q$, $\langle P,Q \rangle_0\neq 0$ if, and only if, there exists $n \in \mathbb{N}$ such that $P=Q=\tun^n$.
Consequently, the kernel of the pairing $\langle-,-\rangle_0$ is the ideal generated by plane posets which are not equal to $\tun$.
\end{enumerate}\end{prop}

\begin{proof} 1. By definition of $\Delta_0$ :
$$\Delta_0(P)\sum_{\substack{\mbox{\scriptsize $I$ biideal of $P$},\\h_{P\setminus I}^I=0}}(P\setminus I) \otimes I.$$
Let $I$ be a biideal of $P$ such that $h_{P\setminus I}^I=0$. Let $x \in P\setminus I$ and $y\in I$. As $I$ is a biideal, $x> y$ is not possible,
so $x<y$. As $h_{P\setminus I}^I=0$, $x<_h y$ is not possible, so $x<_r y$: finally, $P=(P\setminus I)I$. 
Conversely, if $P=P_1P_2$, then $P_2$ is a biideal of $P$, and $h_{P\setminus P_2}^{P_2}=h_{P_1}^{P_2}=0$. \\

2. $\Longleftarrow$. We obviously have $\iota(\tun^n)\leq \tun^n$ and $\phi(\tun^n,\tun^n)=0$, so $\langle \tun^n,\tun^n\rangle_0=1$.\\

$\Longrightarrow$. Let us assume that $\langle P,Q \rangle_q\neq 0$. Then $P$ and $Q$ have the same degree, which we denote by $n$.
Moreover, $\iota(P)\leq Q$ and $\phi(P,Q)=0$. Let $x',y'\in Q$. We put $x'=\theta_{P,Q}(x)$ and $y'=\theta_{P,Q}(y)$.
If $x'<_h y'$ in $Q$, as $\iota(P)\leq Q$ necessarily $x<_r y \in P$, so:
$$\phi(P,Q)\geq \sharp\{(x,y) \in P^2\mid x<_r y \mbox{ and }\theta_{P,Q}(x)<_h \theta_{P,Q}(y)\}>0.$$
This is a contradiction. Hence, if $x',y'$ are two different vertices of $Q$, they are not comparable for $\leq_h$: $Q=\tun^n$.
Symmetrically, $P=\tun^n$. \end{proof}

\subsection{Level of a plane poset}

\begin{defi} 
Let $P$ be a plane poset. Its \emph{level} is the integer:
$$\ell(P)=\sharp\{(x,y) \in P^2\mid x<_r y\}.$$
\end{defi}

{\bf Examples.} Here are the level of plane posets of cardinality $\leq 3$:
$$\begin{array}{c|c|c|c|c|c|c|c|c|c}
P&\tun&\tdeux&\tun\tun&\ttroisdeux&\ttroisun&\ptroisun&\tdeux\tun&\tun\tdeux&\tun\tun\tun\\
\hline \ell(P)&0&0&1&0&1&1&2&2&3
\end{array}$$

\begin{prop}
Let $P$ and $Q$ be two plane posets. If $P\leq Q$, then any path in the Hasse (oriented) graph of the poset $(\PP,\leq)$ from $P$ to $Q$
has length $\ell(Q)-\ell(P)$.
\end{prop}

\begin{proof}  {\it First step.} Let $R$ be  a plane poset of cardinality $n$. Then:
\begin{equation} \label{E1}
\ell(R)=\sharp\{(x,y) \in R^2\mid x<y\}-\sharp\{(x,y) \in R^2\mid x<_h y\}=\frac{n(n-1)}{2}-\sharp E(R).
\end{equation}
Consequently, if $R<S$ in $\PP(n)$, then $E(S) \subsetneq E(R)$, so $\ell(R)<\ell(S)$.\\

{\it Second step.} Let $R$ and $S$ be two plane posets of the same cardinality $n$, such that there is an edge from $R$ to $S$ in the Hasse graph of $(\PP(n), \leq)$.
Then $R < S$, so $E(S)\subsetneq E(R)$. Let us put $k=\sharp E(R)-\sharp E(S)$. Note that $k\geq 1$. By the first step, $\ell(S)-\ell(R)=k$.
If $k\geq 2$, by lemma \ref{13}, there exists $P_1,\ldots,P_{k-1}\in \PP(n)$, such that $R\preceq P_1\preceq \ldots \preceq P_{k-1} \preceq S$.
Consequently, $R<P_1<\ldots<P_{k_1}<S$, so there is no edge from $R$ to $S$ in the Hasse graph: contradiction. So $k=1$. \\

{\it Conclusion}. Let $P=P_0<P_1<\ldots<P_{k-1}<P_k=Q$ be a path from $P$ to $Q$ in the Hasse graph. For all $0\leq i \leq k-1$,
there is an edge from $P_i$ to $P_{i+1}$, so $\ell(P_{i+1})=\ell(P_i)+1$ from the second step. Finally, $\ell(S)=\ell(R)+k$. \end{proof}

\begin{prop}
Let $P,Q \in \PP(n)$, such that $\iota(P)\leq Q$. Then:
$$\langle P,Q \rangle_q=q^{n(n-1)-\ell(P)-\ell(Q)}.$$
\end{prop}

\begin{proof}
As $\iota(P) \leq Q$, for all $x,y\in P$, $\theta_{P,Q}(x)<_h \theta_{P,Q}(y)$ in $Q$ implies that $x<_r y$ in $P$.
So, with the help of (\ref{E1}):
\begin{eqnarray*}
\sharp\{(x,y)\in P^2\mid x<_r y \mbox{ and }\theta_{P,Q}(x)<_h \theta_{P,Q}(y)\}
&=&\sharp\{(x,y)\in P^2\mid \theta_{P,Q}(x)<_h \theta_{P,Q}(y)\}\\
&=&\sharp\{(x',y') \in Q^2\mid x'<_h y'\}\\
&=&\frac{n(n-1)}{2}-\ell(Q).
\end{eqnarray*}
Moreover, $\iota(Q)\leq \iota^2(P)=P$, so, for all $x,y\in P$, $x\leq_h y$ in $P$ implies that $\theta_{P,Q}(x)<_r \theta_{P,Q}(y)$ in $Q$.
So, with the help of (\ref{E1}):
\begin{eqnarray*}
\sharp\{(x,y)\in P^2\mid x<_h y \mbox{ and }\theta_{P,Q}(x)<_r \theta_{P,Q}(y)\}
&=&\sharp\{(x,y) \in P^2\mid x<_h y\}\\
&=&\frac{n(n-1)}{2}-\ell(P).
\end{eqnarray*}
Summing, we obtain $\phi(P,Q)=n(n-1)-\ell(P)-\ell(Q)$. \end{proof}

\bibliographystyle{amsplain}
\bibliography{biblio}

\end{document}